\documentclass[10pt]{amsart}
\usepackage{amssymb}

\newtheorem{theorem}{Theorem}
\newtheorem{lemma}[theorem]{Lemma}
\newtheorem{proposition}[theorem]{Proposition}

\newtheorem{definition}[theorem]{Definition}

\newcommand{\newsection}[1] {\section{#1}\setcounter{theorem}{0}
 \setcounter{equation}{0}\par\noindent}

\newcommand{\R}{{\mathbb R}}
\newcommand{\Z}{{\mathbb Z}}

\renewcommand{\phi}{\varphi}

\newcommand{\lcal}{\mathcal{L}}

\newcommand{\rcal}{\mathcal{R}}

\newtheorem{theo}{{\sc Theorem}}

\newtheorem{lem}[theo]{{\sc Lemma}}
\newtheorem{prop}[theo]{{\sc Proposition}}

\newcommand{\varepislon}{\varepsilon}

\begin{document}

\title[Blowup of quasimodes]
{About the blowup of quasimodes \\ on Riemannian manifolds}
\thanks{The authors were supported by the National Science Foundation,
Grants DMS-0555162, DMS-0099642, and DMS-0354668.}
\author{Christopher D. Sogge}
\author{John A. Toth}
\author{Steve Zelditch}
\address{Department of Mathematics, Johns Hopkins University,
Baltimore, MD 21218}
\address{Department of Mathematics, McGill University, Montreal,
Candada}
\address{Department of Mathematics, Johns Hopkins University,
Baltimore, MD 21218}
\email{sogge@jhu.edu}
\email{jtoth@math.mcgill.ca}
\email{zeldtich@math.jhu.edu}

\maketitle

\begin{abstract} On any compact Riemannian manifold $(M, g)$ of dimension $n$, the $L^2$-normalized eigenfunctions ${\phi_{\lambda}}$ satisfy $||\phi_{\lambda}||_{\infty} \leq C \lambda^{\frac{n-1}{2}}$ where $-\Delta \phi_{\lambda} = \lambda^2 \phi_{\lambda}.$ The bound is sharp in the class of all $(M, g)$ since it is obtained by zonal spherical harmonics on the standard $n$-sphere $S^n$. But of course, it is not sharp for many Riemannian manifolds, e.g. flat tori $\R^n/\Gamma$. We say that $S^n$, but not $\R^n/\Gamma$, is a Riemannian manifold with maximal eigenfunction growth. The problem which motivates this paper is to determine the $(M, g)$ with maximal eigenfunction growth. In an earlier work, two of us showed  that such an $(M, g)$ must have a point $x$ where the set ${\mathcal L}_x$ of geodesic loops at $x$ has positive measure in $S^*_x M$.  We strengthen this result here by showing that such a manifold must have a point where the set ${\mathcal R}_x$ of recurrent 
directions for the geodesic flow through x satisfies $|{\mathcal R}_x|>0$.  We also show that
if there are no such points, $L^2$-normalized quasimodes have sup-norms that are $o(\lambda^{n-1)/2})$, and, in the other extreme, we show that if there is a point blow-down $x$ at which the first return map for the flow is the identity, then there is a sequence
of quasi-modes with $L^\infty$-norms that are $\Omega(\lambda^{(n-1)/2})$.
\end{abstract}

\newsection{Introduction}\label{section1}

In a recent series of articles \cite{SZ, TZ, TZ2, TZ3}, the
authors have been studying the relations between dynamics of  the
geodesic flow and
 $L^p$ estimates of $L^2$-normalized
eigenfunctions of the Laplacian on a compact Riemannian manifold
$(M, g)$. The general aim is to understand how the behavior of
geodesics modifies the universal estimates of $L^{\infty}$ of
Avakumovic-Levitan-H\"ormander, and the  general $L^p$ norms
obtained by Sogge \cite{So1} (see also \cite{KTZ}  and \cite{SS} for recent and
more general results).  In particular, we wish to characterize the global
dynamical properties of the geodesic flow of
 $(M, g)$ which exhibit
extremal behavior of eigenfunction growth. This problem is an example of  global
analysis of eigenfunctions as surveyed in  \cite{Z3}.

This article continues the series. Its purpose is to sharpen the
previous  results on   maximal eigenfunction growth and to prove
they are sharp by giving converse results.  To introduce our
subject, we need some notation.
 Let $\{-\lambda_\nu^2\}$ denote the eigenvalues of
$\Delta$, where $0\le \lambda_0^2\le \lambda^2_1\le \lambda^2_2\le
\dots$ are counted with multiplicity and let $\{\phi_{\lambda_\nu(x)}\}$ be
an associated orthonormal basis of $L^2$-normalized eigenfunctions
(modes).  If $\lambda^2$ is in the spectrum of $-\Delta$, let
$V_{\lambda} = \{\phi: \Delta \phi = -\lambda^2 \phi\}$ denote the
corresponding eigenspace.  We measure the growth rate of
$L^{\infty}$-norms of modes by
 \begin{equation} \label{LWL} L^{\infty}(\lambda, g) = \sup_{\phi\in V_{\lambda}: ||\phi||_{L^2}
= 1 }
 ||\phi||_{L^{\infty}}. \end{equation}
The general result of \cite{Le, A} is that
\begin{equation}\label{UNIVBOUND}
L^{\infty}(\lambda, g) = O (\lambda^{\frac{n-1}{2}}).
\end{equation}

If  this bound is achieved for some subsequence of eigenfunctions,
i.e., $L^\infty(\lambda,g)=\Omega(\lambda^{(n-1)/2})$,  we say
that $(M,g)$ has {\it maximal eigenfunction growth}. The
corresponding
 sequence of $L^2$-normalized
eigenfunctions $\{\phi_{\lambda_\nu}\}$  have  $L^{\infty}$ norms which
are comparable to those of  zonal spherical harmonics on $S^n$.
The main result of \cite{SZ} is a necessary condition on maximal
eigenfunction growth:  there must then exist a point  $z$ such
that a positive measure of geodesics emanating from $z$ return to
it at a fixed time $T$.   In the case where all directions loop
back, we will call $z$ a {\it blow-down point} since the natural
projection $\pi: S^* M \to M$ has a blow-down singularity on
$S^*_z M$. For lack of a standard term, in  the general case of a
positive measure of loops we call $z$ a {\it partial blow-down
point}. Examples of blow-down points are poles of surfaces of
revolution and umbilic points of two-dimensional tri-axial
ellipsoids. In the case of surfaces of revolution, all geodesics
emanating from poles smoothly close up while in the case of
ellipsoids, the geodesics emanating from umbilic points loop back
but with two exceptions do not close up smoothly. One can
construct partial blow-down points by perturbing these metrics in
small polar caps to obstruct some of the geodesics.

A comparison of surfaces of revolution and ellipsoids shows that
the necessary condition on maximal eigenfunction growth in
\cite{SZ} is not sharp and focusses attention on the
distinguishing
 dynamical
invariant. It is easily seen that  surfaces of revolution are of
maximal eigenfunction growth (cf. e.g. \cite{SZ}) and that zonal
eigenfunctions achieve the sup norm bound (\ref{UNIVBOUND}) at the
poles (see e.g. \cite{So2}). However, ellipsoids are not of maximal eigenfunction
growth. It is proved in \cite{T1,T2} that obvious analogues of
zonal eigenfunctions on an ellipsoid only have the growth rate
$\frac{\lambda^{\frac{1}{2}}}{\sqrt{\log \lambda}}$, and as a consequence
 of Theorem \ref{RECURRENCE}, it follows that $\| \phi_{\lambda_\nu} \|_{\infty} = o(\lambda_{\nu}^{1/2})$ on such a surface. Indeed, although we will not prove it here, it is likely that $ L^{\infty}(\lambda,g) = {\mathcal O} \left( \frac{\lambda^{\frac{1}{2}}}{\sqrt{\log \lambda}} \right)$ on the ellipsoid.

The obvious difference between the geodesics of surfaces of
revolution and ellipsoids is in the nature of the {\it first
return map} $\Phi_z$ on directions $\theta \in S^*_z M$. This map is
simplest to define when $z$ is a blow-down point, i.e. if all
directions loop back. The first return map is then the fixed time
map of the geodesic flow $G^t$ acting on the sphere bundle:
\begin{equation} \label{RM} \Phi_z =
G^T_z: S^*_z M \to S^*_z M. \end{equation} Below we will define it
in the case of a partial blow-down point.

 In  the case
of a surface of revolution  $G^T_z
 = id$  is the identity map on
$S^*_z M$, while in the case of an ellipsoid it has just two fixed
points, one attracting and one repelling, and all directions
except for the repelling fixed point are in the basin of
attraction of the attracting fixed point. This comparison
motivates the first theme of the present article: to  study of the
relation of maximal eigenfunction growth and the dynamics of this
first return map. The relevance of the first return map to
problems in spectral theory was already observed by
 Y. Safarov et al in studying
clustering in the spectrum \cite{S, GS, SV}. It seems reasonable
to conjecture that maximal eigenfunction growth can only arise in
the identity case, or at least
 when
$\Phi_z$ has a positive measure of fixed points, i.e. if there
exists a positive measure of smoothly closed geodesics through a
point $z$ of $M$.

However, this necessary condition is not sufficient.
 A counterexample was constructed in \cite{SZ}  of a
surface with a positive measure of closed geodesics through a
point $z$ but which does not have maximal eigenfunction growth.
Further,  we conjecture that  even existence of blow-down points
fails to be sufficient: for instance, every point is a blow-down
point on a Zoll manifold, but we conjectue that generic  Zoll
manifolds fail to
 have maximal eigenfunction growth.

 We can close the gap between necessary and sufficient
conditions on eigenfunction growth by generalizing the problem of
eigenfunction growth to include approximate eigenfunctions, or
{\it quasi-modes}.  As we shall see, the widest collection that
one could hope to have pointwise $o(\lambda^{(n-1)/2})$
upperbounds are defined as follows:

\begin{definition}\label{quasidef}
A
sequence $\{\psi_\lambda\}$, $\lambda=\lambda_j$, $j=1,2,\dots$ is
a sequence of admissible quasimodes if $\|\psi_\lambda\|_2=1$ and
\begin{equation}\label{admissible}
\|(\Delta+\lambda^2)\psi_\lambda\|_2+\|S^\perp_{2\lambda}\psi_\lambda\|_\infty
= o(\lambda).
\end{equation}
\end{definition}


Here, $S_\mu^\perp$ denotes the projection onto the $[\mu,\infty)$
part of the spectrum of $\sqrt{-\Delta}$, and in what follows
$S_\mu=I-S_\mu^\perp$, i.e., $S_\mu f=\sum_{\lambda_j< \mu}
e_j(f)$, where $e_j(f)$ is the projection of $f$ onto the
eigenspace with eigenvalue $\lambda_j$.

There are many notions of quasimodes in the literature, but the
above one seems to be new. We shall describe why it seems to give
the natural class of ``approximate eigenfunctions" in theorems
that say maximal pointwise blowup implies the existence of certain
types of dynamics (see below). We should also point out that the
technical condition in the definition that
$\|S^\perp_{2\lambda}\psi_\lambda\|_\infty=o(\lambda)$ is
typically not included in the definitions of quasimodes.  We need
it for some of our results, and we note that, by Sobolev, when the
dimension is smaller than $4$, it is a consequence of the main
part of the definition, i.e.,
$\|(\Delta+\lambda^2)\psi_\lambda\|_2=o(\lambda)$. A model case of
functions satisfying \eqref{admissible} would be a sequence of
$L^2$-normalized functions $\{\psi_{\lambda_j}\}$ whose
$\sqrt{-\Delta}$ spectrum lies in intervals of the form
$[\lambda_j-o(1), \lambda_j+o(1)]$ as $\lambda_j\to \infty$.


It is  natural to consider quasi-modes because the methods of
producing blowup apply in fact to quasi-modes. Results on modes
are obtained only by specializing results on quasi-modes. In
examples where one knows the eigenfunctions in detail, such as
surfaces of revolution or ellipsoids, the reason is usually that
the modes and quasi-modes are the same.  We should also point out
that while \eqref{admissible} is a natural condition for classes
of ``approximate eigenfunctions" satisfying
$o(\lambda^{((n-1)/2})$ sup-norm upperbounds, it is not
necessarily so for lowerbounds. Indeed, the quasimodes satisfying
$\Omega(\lambda^{(n-1)/2})$ lowerbounds that we shall construct
will satisfy $\|(\Delta+\lambda^2)\psi_\lambda\|_2=o(1)$ (quasimodes of
order zero), which is
weaker than \eqref{admissible}.

 An important  example of such a
quasi-mode is a sequence of ``shrinking spectral projections",
i.e. the $L^2$-normalized projection kernels
$$\Phi_j^z(x) = \frac{\chi_{[\lambda_j, \lambda_j + \epsilon_j]}(x,
z)}{\sqrt{\chi_{[\lambda_j, \lambda_j + \epsilon_j]}(z, z)}}$$
with second point frozen at a point $z \in M$ and with width
$\epsilon_j \to 0$. Here, $\chi_{[\lambda_j, \lambda_j +
\epsilon_j]}(x, z)$ is the orthogonal projection onto the sum of
the eigenspaces $V_{\lambda}$ with $\lambda \in [\lambda_j,
\lambda_j + \epsilon_j]$ The zonal eigenfunctions of a surface of
revolution are examples of such shrinking spectral projections for
a sufficiently small $\epsilon_j$, and when $z$ is a partial focus
such $\Phi_j^z(x)$ are generalizations of zonal eigenfunctions. On
a general Zoll manifold, shrinking spectral projections of widths
$\epsilon_j = O(\lambda_j^{-1})$ are the direct analogues of zonal
spherical harmonics, and they would satisfy the analog of
\eqref{admissible} where $o(\lambda)$ is replaced by the much
stronger $O(\lambda^{-1})$.


\subsection{General results}

Our first result is quite general. It shows that the main result
of \cite{SZ} extends to admissible quasi-modes and also gives a
reasonable converse.

\begin{theo} \label{QMB} Let $(M^n, g)$ be a compact Riemannian manifold with Laplacian $\Delta$. Then:
\begin{enumerate}

\item If there exists an admissible sequence of  quasi-modes with
$||\psi_{\lambda_k}||_{L^{\infty}} = \Omega(\lambda_k^{\frac{n-1}{2}})$, then
there exists a partial blow-down point $z \in M$ for the geodesic
flow. If $(M, g)$ is real-analytic, then there exists a blow-down
point.

\item Conversely, if there exists a blow-down point and if the
first return  map is the identity, $G^T_z = id,$ then    there
exists a quasi-mode sequence $\{\psi_{\lambda_k}\}$ of order $0$
with $||\psi_{\lambda_k}||_{L^{\infty}} = \Omega(\lambda_k^{(n-1)/2})$.

\end{enumerate}

\end{theo}

%

As we mentioned before, Part (1) of the Theorem was proved in
\cite{SZ} for modes.  The improvement here is that there must be a
partial blowdown point if a sequence of quasimodes has maximal
sup-norm growth.  We can make a further improvement and show that
there must be a special type of partial blowdown point, a {\it
recurrent point} for the geodesic flow.

%
%
%
%

Let us be more specific.  Given $x\in M$, we let $\lcal_x$ the set
of loop directions at $x$:
\begin{equation} \lcal_x = \{\xi \in S^*_x M : \exists T: \exp_x T
\xi = x \}.
\end{equation}
Thus, $x$ is a partial blow-down point if $|\lcal_x| > 0$ where
$|\cdot|_z$ denotes the surface measure on $S^*_x M$ determined by
the metric $g_x$. We also let $T_x: S^*_x M \to \R_+ \cup
\{\infty\}$ denote the return time function to $x$,
$$T_x (\xi) = \left\{ \begin{array}{ll} \inf \{ t > 0: \exp_x t \xi
= x\}
, & \; \mbox{if}\; \xi \in \lcal_x; \\ & \\
+ \infty, & \mbox{if no such} \;t \; \mbox{exists}. \end{array}
\right. $$

The first return map is thus
$$G^{T_x}_x : \lcal_x \to S^*_x M. $$
In the general case, $\lcal_x$ is not necessarily invariant under
$G^T_x$. To obtain  forward/backward  invariant sets we put
\begin{equation} \lcal_x^{\pm  \infty} = \bigcap_{\pm k \geq 0}
(G^{ T_x}_x)^k \lcal_x,  \end{equation} and also put
$\lcal_x^{\infty} = \lcal_x^{+\infty} \cap \lcal_x^{- \infty}$.
Then  $(\lcal_x^{\infty}, G_x^{T_x})$ defines a dynamical system.
We equip it with the restriction of the surface measure
$|\cdot|_x$, but of course this measure is not generally invariant
under $G_x^{T_x}$. We further define the set of {\it recurrent
loop directions} to be the subset
$$\rcal_x = \{\xi \in \lcal_x^{\infty}: \xi \in
\omega(\xi)\}, $$ where $\omega(\xi)$ denotes the $\omega$-limit
set, i.e. the limit points of the orbit $\{(G_x^{T_x})^n \xi: n
\in \Z_+ \}.$  Equivalently, $\xi\in \lcal_x^\infty$ belongs to
$\rcal_x$ if infinitely many iterates, $(G^{T_x}_x)^n\xi$, $n\in
\Z$, belong to $\Gamma$, whenever $\Gamma$ is a neighborhood of
$\xi$ in $S^*_xM$.  Finally, we say that $x$ is a {\it recurrent
point} for the geodesic flow if $|\rcal_x|>0$.

Our improvement of the first half of the preceding theorem will be
based on the following result that will give upperbounds for
admissible quasimodes under a natural dynamical assumption.

\begin{theo}\label{RECURRENCE}  Suppose that $|\rcal_x|=0$ for every $x\in M$.
Then,
given $\varepsilon>0$, one can find
$\Lambda(\varepsilon)<\infty$ and $\delta(\varepsilon)>0$ so that
\begin{equation}\label{a}
\|\chi_{[\lambda,\lambda+\delta(\varepsilon)]}f\|_{L^\infty(M)}\le
\varepsilon \lambda^{(n-1)/2}\|f\|_{L^2(M)}, \quad \lambda\ge
\lambda(\varepsilon).
\end{equation}
Under the stronger hypothesis that $|{\mathcal L}_x|=0$ for every
$x\in M$ one has that for every $\delta>0$ there is a
$\lambda(\delta)$ so that
\begin{equation}\label{1'}
\|\chi_{[\lambda,\lambda+\delta]}f\|_{L^\infty(M)}\le
C\delta^{1/2}\lambda^{(n-1)/2}\|f\|_{L^2(M)}, \quad \lambda\ge
\lambda(\delta),
\end{equation}
for some constant $C=C(M,g)$ which is independent of $\delta$ and
$\lambda$.
\end{theo}

 To show that Theorem \ref{RECURRENCE} is indeed stronger than
 the result in \cite{SZ}, we note that
 it is well-known \cite{T1} that Liouville metrics on spheres
 (such as the triaxial ellipsoid) satisfy the condition
 $|\rcal_{x}| = 0$ for each $x \in M$. However, $|\lcal_{z}| =1$
when $z \in M$ is an umbilic point for the metric and so, Theorem
 \ref{RECURRENCE} applies to these examples as well  and  gives the $L^{\infty}(\lambda,g) = o(\lambda^{(n-1)/2})$-bound.

 Also, as noted in \cite{SZ}, if one uses an
 interpolation argument involving the estimates in \cite{So1}, then
 \eqref{a} implies that $L^p$-estimates are not saturated for $p>2(n+1)/(n-1)$
 under the assumption that $|{\mathcal R}_x|=0$ for all $x\in M$.  Recently,
 one of \cite{kakeya} has formulated a sufficient condition for the non-saturation
 of $L^p$-estimates in dimension two for $2<p<6$ that involves the concentration along
 geodesics.  A condition for the endpoint case of $p=6$ for dimension 2 or
 $2<p<2(n+1)/(n-1)$ or $p=2(n+1)/(n-1)$ in higher dimensions remains open.

A corollary of Theorem \ref{RECURRENCE} will be given in Theorem
\ref{quasimodetheorem} below which says that if there is a
sequence of admissible quasimodes with maximal sup-norm blowup,
then there must be a recurrent point for the geodesic flow.
Equivalently, if there is no such point, then a sequence of
admissible quasimodes must have sup-norms that are
$o(\lambda^{(n-1)/2})$.

We can write the conclusion of Theorem \ref{RECURRENCE} in the
shorthand notation
\begin{equation}\label{c}\|\chi_{[\lambda,\lambda+o(1)]}\|_{L^2(M)\to L^\infty(M)}=
o(\lambda^{(n-1)/2}).\end{equation} This result is optimal in one
sense because the well known sup-norm estimate
$$\|\chi_{[\lambda,\lambda+1]}\|_{L^2(M)\to
L^\infty(M)}=O(\lambda^{(n-1)/2})$$ cannot be improved on ${\it
any}$ compact Riemannian manifold (see e.g., \cite{Soggebook}),
and this together with \eqref{c} provides a motivation for
Definition \ref{quasidef}. On the other hand, it might be the case
that \eqref{c} holds under the weaker hypothesis that $|{\mathcal
C}_x|=0$ for every $x\in M$, if ${\mathcal C}_x\subset \lcal_x$
denotes the set of periodic directions, i.e. initial directions
for {\it smoothly closed geodesics} through $x$.  Also, because of
the sharp Weyl formula, the bounds in \eqref{1'} are clearly sharp
in the sense that one cannot take a larger power of $\delta$, and
one also needs the hypothesis that $\lambda$ is large depending on
$\delta$.

We should point out that Theorem \ref{RECURRENCE} is related to
the error estimates for the Weyl law of Duistermaat and Guillemin
\cite{DG} and Ivrii \cite{Iv1} and the error estimates of Safarov
\cite{S} for a local Weyl law.  Like Ivrii's argument, ours are
just based on exploiting the nature of the singularity of the wave
kernel $e^{it\sqrt{-\Delta}}$ at $t=0$.  Unlike these other works,
though, we can prove our main estimate, \eqref{a}, without using
Tauberian lemmas.  Traditionally, sup-norm estimates like
\eqref{a} were obtained by deducing them from stronger asymptotic
formulas, e.g. appropriate Weyl laws with remainder bounds.  Two
of us in \cite{SZ} used this approach to prove the weaker variant
of Theorem \ref{RECURRENCE} where one deduces \eqref{a} under the
stronger assumption that there are no partial blowdown points.  In
the present work, we are able to prove these stronger results
using a simpler argument that yields the main estimate \eqref{a}
directly but does not seem to yield a correspondingly strong local
Weyl law under the assumption that $|\rcal_x|=0$ for all $x$.

\subsection{Invariant tori and surfaces with maximal eigenfunction growth }

An easy consequence  of our results is the following:

\begin{theo} \label{ergsup} If a   real analytic  Riemannian
$n$-manifold  $(M, g)$ has  maximal growth of eigenfunctions or
admissible quasi-modes, then its  geodesic flow has an  invariant
Lagrangian submanifold $\Lambda \simeq S^1 \times S^{n-1} \subset
S^*_g M$. Hence, a surface
 with ergodic geodesic   never has maximal growth of eigenfunctions or admissible
 quasi-modes.
\end{theo}

By \cite{SZ}, a real analytic surface with maximal eigenfunction
growth must be a topological sphere, so the last result only adds
new information when $M \simeq S^2$.    Real analytic ergodic
metrics on $S^2$  have been constructed by K. Burns - V. Donnay
\cite{BD} and by Donnay-Pugh \cite{DP, DP2}. There even exist such
surfaces embedded in $\R^3$ (see \cite{BD} for computer graphics
of such surfaces). Note that such metrics must have conjugate
points, so the logarithmic estimates of \cite{Be} do not apply.

\newsection{Recurrent points and upperbounds for quasimodes}\label{section2}

We shall first prove Theorem \ref{RECURRENCE}.  To do so,
we first note that
\begin{equation}\label{b}\|\chi_{[\lambda,\lambda+\delta]}\|_{L^2(M)\to
L^\infty(M)}^2=\sup_{x\in M}\sum_{\lambda_j\in
[\lambda,\lambda+\delta]}|e_j(x)|^2,\end{equation} if $\{e_j\}$ is
an orthonormal basis of eigenfunctions with eigenvalues
$\{\lambda_j\}$.
 By compactness, we conclude that  the first inequality in Theorem \ref{RECURRENCE} follows from the
following local version

%

\begin{proposition}\label{prop1.2}  Suppose that $x_0\in M$ satisfies
$|\rcal_{x_0}|=0$.  Then, given $\varepsilon>0$ we can find a
neighborhood ${\mathcal N}_\varepsilon$ of $x_0$, a
$\lambda_\varepsilon<\infty$ and a $\delta(\varepsilon)>0$ so that
\begin{equation}\label{d}\sum_{\lambda_j\in [\lambda,\lambda+\delta(\varepsilon)]}
|e_j(x)|^2\le \varepsilon^2 \lambda^{(n-1)/2}, \, \, x\in
{\mathcal N}_\varepsilon, \, \, \lambda\ge
\lambda_\varepsilon.\end{equation}
\end{proposition}

In \cite{SZ} we exploited the lower semicontinuity of $L(x,\xi)$
where $L(x,\xi)$ equaled the shortest loop in the direction
$\xi\in S^*_xM$ if there was one and $L(x,\xi)=+\infty$ if not.

To prove our improvement of the main result of \cite{SZ}, instead
of watching all loops, we shall just watch all loops of a given
length length $\ell$ and initial direction $\xi$ which have the
property that $\text{dist }(G^\ell_x(\xi),\xi)\le \delta$, with,
$\text{dist }(\cdot \, , \, \cdot)$ being the distance induced by
the metric, and, as before, $G^\ell_x(\xi)$ being the terminal
direction. So we let $L_\delta(x,\xi)$ be the length of the
shortest such loop fulfilling this requirement if it exists and
$+\infty$ otherwise. Then $L_\delta(x,\xi): S^*_xM\to (0,+\infty]$
is lower semicontinuous and $1/L_\delta(x,\xi)$ is upper
semicontinuous. We then let $\rcal_x^\delta$ then is all $\xi$ for
which $1/L_\delta(x,\xi)\ne 0$.

To exploit this, if $x_0$ is as in the proposition, we shall
choose $\delta$ large enough so that
$|\rcal_x^\delta|<\varepsilon^2/2$ and then take $f(x,\xi)$ to be
$1/L_\delta(x,\xi)$ in the following variant of Lemma 3.1 in
\cite{SZ}. We shall take the parameter $\rho$ in the lemma to be
$1/10T$ where $T$ is much larger than $1/\delta^{(n-1)}$.

\begin{lemma}\label{locallemma} Let $f$ be a nonnegative upper semicontinuous
function on ${\mathcal O}\times S^{n-1}$, where ${\mathcal
O}\subset {\mathbb R}^n$ is open.  Fix $x_0\in {\mathcal O}$ and
suppose that $\{\xi\in S^{n-1}: \, f(x_0,\xi)\ne 0\}$ has measure
$\le \varepsilon/2$, with $\varepsilon>0$ being fixed.  Let
$\rho>0$ be given.  Then there is a neighborhood ${\mathcal N}$ of
$x_0$ an open set $\Omega_b\subset S^{n-1}$ satisfying
\begin{align*}
|f(x,\xi)|&\le \rho, \, \, (x,\xi)\in {\mathcal N}\times
S^{n-1}\backslash \Omega_b
\\
|\Omega_b|&\le \varepsilon.
\end{align*}
Furthermore, there is a $b(\xi)\in C^\infty$ supported in
$\Omega_b$ satisfying $0\le b\le 1$, and having the property that
if $B(\xi)=1-b(\xi)$ then $f(x,\xi)\le \rho$ on ${\mathcal
N}\times \text{supp} B$.
\end{lemma}

{\noindent Proof:}  The proof is almost identical to Lemma 3.1 in
\cite{SZ}.

By assumption the set $E_\rho=\{\xi\in S^{n-1}: f(x_0,\xi)\ge
\rho\}$ satisfies $|E_\rho|\le \varepsilon/2$.  Let
$$E_\rho(j)=\{\xi\in S^{n-1}: \, f(x,\xi)\ge \rho, \, \text{some
}x\in \overline{B}(x_0,1/j)\},$$ where $\overline{B}(x_0,r)$ is
the closed ball of radius $r$ about $x_0$.  Then clearly
$E_\rho(j+1)\subset E_\rho(j)$.  Also, if $\xi\in
\cap_{j\ge1}E_\rho(j)$ then for all $j$ one can find $x_j\in
\overline{B}(x_0,1/j)$ such that $f(x_j,\xi)\ge \rho$, which means
that
$$\rho\le \limsup_{j\to\infty}f(x_j,\xi)\le f(x_0,\xi),$$
by the upper semicontinuity of $f$.  Thus,
$$\cap_{j\ge1}E_\rho(j)\subset E_\rho.$$
Consequently, if $j$ is large $|E_\rho(j)|<\varepsilon$.  Fix such
a $j=j_0$ and choose an open set $\Omega_b$ satisfying
$E_\rho(j)\subset \Omega_b$ and $|\Omega_b|<\varepsilon$.  Then
clearly $f(x,\xi)<\rho$ if $(x,\xi)\in B(x_0,1/j_0)\times
S^{n-1}\backslash \Omega_b$.

For the last part, note that the argument we have just given will
show that the sets $E_\rho(j)$ are closed because of the upper
semicontinuity property of $f$.  Thus, if $E_\rho(j_0)$ and
$\Omega_b$ are chosen as above, we need only apply the $C^\infty$
Urysohn lemma to find a smooth function $b(\xi)$ supported in
$\Omega_b$ with range $[0,1]$ and satisfying $b(\xi)=1$, $\xi\in
E_\rho(j_0)$, which then will clearly have the required
properties. \qed

To apply this lemma we first choose a coordinate patch ${\mathcal
K}$ with coordinates $\kappa(x)$ around $x_0$, which we identify
with an open subset of ${\mathbb R}^n$.  Also, fix a number
$\delta>0$ small enough so that $|\rcal_{x_0}^\delta|\le
\varepsilon^2/2$.  We then let $f(x,\xi)$ denote the image of
$1/L_\delta(x,\xi)$ in the induced coordinates for $\{(x,\xi)\in
S^*M:\, x\in {\mathcal K}\}$.  Then, given a large number $T$ (to
be specified later), we can find a function $b\in
C^\infty(S^{n-1})$ with range $[0,1]$ so that
\begin{equation}\label{e}\int_{S^{n-1}}b(\xi)\, d\xi \le
\varepsilon^2,\end{equation} and
\begin{equation}\label{f}
L_\delta(x,\xi) \ge 2T \quad \text{on } \, {\mathcal N}\times
\text{supp }B,
\end{equation}
where ${\mathcal N}\subset \kappa({\mathcal K})$ is a neighborhood
of $x_0$ and
$$B(\xi)=1-b(\xi).$$

Choose a function $\psi\in C^\infty({\mathbb R}^n)$ with range
$[0,1]$ which vanishes outside of ${\mathcal N}$ and equals one in
a small ball centered at $\kappa(x_0)$.  Using these functions we
get zero-order pseudo-differential operators on ${\mathbb R}^n$ by
setting
$$\tilde b(x,D)f(x)=\psi(x)(2\pi)^{-n}\iint
e^{i(x-y)\cdot\xi}b(\xi/|\xi|)\psi(y)f(y)\, dyd\xi,$$ and
$$\tilde B(x,D)f(x)=\psi(x)(2\pi)^{-n}\iint
e^{i(x-y)\cdot\xi}B(\xi/|\xi|)\psi(y)f(y)\, dyd\xi.$$ Note that
both variables of the kernels of these operators have support in
${\mathcal K}$.  If we let $b(x,D)$ and $B(x,D)$ in $\Psi^0(M)$ be
the pullbacks of $\tilde b$ and $\tilde B$, respectively, then
$$b(x,D)+B(x,D)=\psi^2(x).$$
Since
$\psi^2\chi_{[\lambda,\lambda+\delta]}=b(x,D)\chi_{[\lambda,\lambda+\delta]}
+ B(x,D)\chi_{[\lambda,\lambda+\delta]}$ it is clear that
\eqref{d} would follow if we could show that there is a
$T=T(\varepsilon)>S$, $\lambda(\varepsilon)<\infty$ and
$\delta(\varepsilon)>0$ so that
\begin{equation}\label{g}
\|b\chi_{[\lambda,\lambda+\delta(\varepislon)]}\|_{L^2\to
L^\infty} \le C\varepsilon \lambda^{(n-1)/2}, \quad \lambda\ge
\lambda(\varepsilon),
\end{equation}
and
\begin{equation}\label{h}
\|B\chi_{[\lambda,\lambda+\delta(\varepislon)]}\|_{L^2\to
L^\infty} \le C\varepsilon \lambda^{(n-1)/2}, \quad \lambda\ge
\lambda(\varepsilon),
\end{equation}
for some uniform constant $C$ which is independent of
$\varepsilon$.

%

Note that
\begin{align}\label{7'}
\|b\chi_{[\lambda,\lambda+\delta(\varepsilon)]}\|^2_{L^2\to
L^\infty}&=\sup_x \sum_{\lambda_j\in
[\lambda,\lambda+\delta(\varepsilon)]}|be_j(x)|^2
\\
\|B\chi_{[\lambda,\lambda+\delta(\varepsilon)]}\|^2_{L^2\to
L^\infty}&=\sup_x \sum_{\lambda_j\in
[\lambda,\lambda+\delta(\varepsilon)]}|Be_j(x)|^2 \label{8'}
\end{align}

To exploit this we shall use a standard trick of dominating these
truncated sums by smoothed-out versions in order to use the
Fourier transform and the wave operator.  To this end, we choose
$\rho\in C^\infty_0({\mathbb R})$ which vanishes for $|t|>1/2$ and
satisfies $\hat \rho\ge0$ and $\hat \rho(0)=1$.  If we then take
$T$ to be  a fixed multiple of $1/\delta(\varepsilon)$, we
conclude from \eqref{7'} and \eqref{8'} that \eqref{g} and
\eqref{h}  would follow from showing that if $T=T(\varepsilon)$
and $\lambda(\varepsilon)$ are large, then
\begin{align}\label{7''}
\sum_{j=1}^\infty \bigl(\hat\rho(T(\lambda-\lambda_j))\bigr)^2\,
|be_j(x)|^2 &\le C\varepsilon^2\lambda^{n-1}, \quad \lambda\ge
\lambda(\varepsilon)
\\
\sum_{j=1}^\infty \bigl(\hat\rho(T(\lambda-\lambda_j))\bigr)^2\,
|Be_j(x)|^2 &\le C\varepsilon^2\lambda^{n-1}, \quad \lambda\ge
\lambda(\varepsilon). \label{8''}
\end{align}

To prove these, we shall require the following standard result
which is based on the singularity of the wave kernel restricted to
the diagonal at $t=0$. To state the notation, we let
$U=e^{it\sqrt{\Delta}}$ denote the wave group and $U(t,x,y)$ its
kernel.  Then we need the following result which follows from
Proposition 2.2 in \cite{SZ}.

\begin{lemma}\label{wavecalculation}  Let $(M,g)$ have injectivity
radius $>10$ and let $A(x,D)\in \Psi^0(M)$ be a
pseudo-differential operator of order $0$.  Let $\alpha\in
C^\infty_0({\mathbb R})$ vanishes for $|t|\ge 2$ and satisfies
$\alpha(0)=1$.  Then, if $A_0(x,\xi)$ denotes the principal symbol
of $A$,
\begin{multline}\label{i}
(2\pi)^{-1}\int_{-\infty}^\infty \alpha(t) e^{-i\lambda
t}\bigl(AUA^*\bigr)(t,x,x)\, dt
\\
-(2\pi)^{-n}\lambda^{n-1}\int_{\sum g^{jk}(x)\xi_j\xi_k=1}
|A_0(x,\xi)| d\sigma(\xi) =O(\lambda^{n-2}).
\end{multline}
\end{lemma}

In what follows, we may assume without loss of generality that the
hypothesis on the injectivity radius of $M$ is satisfied.

Note that we can rewrite the left side of \eqref{i} as
\begin{equation}\label{j}
(2\pi)^{-n}\int_{-\infty}^\infty \alpha(t) e^{-i\lambda
t}\bigl(AUA^*\bigr)(t,x,x) dt = \sum_j
\hat\alpha(\lambda-\lambda_j) |Ae_j(x)|^2.
\end{equation}
If we choose $\alpha$ as above so that $\hat \alpha\ge 0$, $\hat
\alpha(0)=1$, we conclude from \eqref{i} and \eqref{j} that
\begin{equation}\label{k}
\sum_{|\lambda_j-\lambda|\le 1}|Ae_j(x)|^2 \le
C\lambda^{n-1}\|A_0(x,\, \cdot\, )\|^2_{L^2(S_x^*M)} +
C_A\lambda^{n-2},\end{equation} where $C$ is independent of
$A=A(x,D)\in \Psi^0(M)$.  This will prove to be a useful estimate
in what follows.

Using \eqref{k} we can get \eqref{7''} if we assume, as we may,
that $T>1$.  For then $(\hat\rho(T(\lambda-\lambda_j)))^2\le
C_N(1+|\lambda-\lambda_j|)^{-N}$ for any $N$, which yields
\eqref{7''} as $\|b(x, \, \cdot\,)\|_{L^2(S_x^*M)}^2\le
C\varepsilon^2$, by \eqref{e}.

To finish the proof of \eqref{a} by proving \eqref{8''}, we first
exploit \eqref{f} to see that we can construct a smooth partition
of unity $1=\sum_k \psi_k(\xi)$ of the unit sphere which consists
of $O(\delta^{-(n-1)})$ terms each of which has range in $[0,2]$
and is supported in a small spherical cap of diameter smaller than
$\delta/10$. We then let $B_k(x,D)$ be the zero-order
pseudo-differential operator whose symbol equals $\tilde B(x,\xi)
\psi_k(\xi/|\xi|)$ in the coordinates used before.  Since $\delta$
is fixed, we would have
\eqref{8''}
 if we could show that
\begin{equation}\label{8'''}
\sum_{j=1}^\infty \bigl(\hat\rho(T(\lambda-\lambda_j))\bigr)^2
|B_ke_j(x)|^2\le C T^{-1}\lambda^{n-1}+C_{B_k,T}\lambda^{n-2}.
\end{equation}
Indeed, if $T$ is chosen large enough so that
$C\delta^{-(n-1)}T^{-1}\le \varepsilon^2$, then, since $B=\sum
B_k$ by applying the Cauchy-Schwarz inequality, we get \eqref{8''}
for large enough $\lambda$. As we shall see, the constant $C$ in
\eqref{8'''} can be taken to be $O(1)$ as $\delta\to 0$; however,
the reduction to estimates for each single $B_k$ contributes an
additional factor $O(\delta^{-(n-1)})$ to the constant in
\eqref{8'}.

To prove \eqref{8'''}, we note that we can rewrite the left side
as
$$\frac1{2\pi}\int_{-\infty}^\infty T^{-1}\bigl(\rho \ast
\rho\bigr)(t/T)\, \bigl(B_kUB_k^*\bigr)(t,x,x) e^{-it\lambda}\,
dt.$$ To estimate this, we need to exploit the fact that our
hypothesis \eqref{f} implies that $(t,x)\to (B_kUB_k^*)(t,x,x)$ is
smooth when $0< |t|\le T$.  Also, by construction,
$(\rho\ast\rho)(t/T)=0$ for $|t|>T$.  To use these facts, we
choose $\beta\in C^\infty_0({\mathbb R})$ satisfying $\beta(t)=1$,
$|t|<1$ and $\beta(t)=0$, $|t|>2$ and then split the left side of
\eqref{8'''} as
\begin{multline*}
\frac1{2\pi}\int \beta(t)
T^{-1}(\rho\ast\rho)(t/T)(B_kUB^*_k)(t,x,x)e^{-i\lambda t}\, dt
\\
+ \frac1{2\pi}\int \bigl(1-\beta(t)\bigr)
T^{-1}(\rho\ast\rho)(t/T)(B_kUB^*_k)(t,x,x)e^{-i\lambda t}\, dt
=I+II.
\end{multline*}

If we integrate by parts we see that $II$ must be
$O(\lambda^{-N})$ for any $N$, which means that we are left with
showing that $I$ enjoys the bounds in \eqref{8'''}.  However,
since we are assuming that $T>1$, one can check that the inverse
Fourier transform of $t\to \beta(t)T^{-1}(\rho\ast\rho)(t/T)$ must
be $\le C_N T^{-1}(1+|\tau|)^{-N}$ for any $N$ if $\tau$ is the
variable dual to $t$.  Thus, for every $N$,
$$I\le
C_N
T^{-1}\sum_{j=1}^\infty(1+|\lambda-\lambda_j|)^{-N}|B_ke_j(x)|^2,$$
which means that our remaining estimate \eqref{8'''} also follows
from \eqref{k}.

One proves \eqref{1'} by the above argument if one takes $\delta$
in the last step to be equal to 1. \qed

\subsection{Blowup rates for quasimodes: Proof of Theorem \ref{QMB} (i)}  Next, we shall
show that we can extend the blowup results of \cite{SZ} for
eigenfunctions to include the admissible quasimodes (defined in
Definition \ref{quasidef}) and also allow one to conclude that
there must points through which there is a positive measure of recurrent directions for
the geodesic flow.

\begin{theorem}\label{quasimodetheorem}  Suppose that $\psi_\lambda$ is a sequence of
admissible quasimodes satisfying
$$\|\psi_\lambda\|_\infty=\Omega(\lambda^{(n-1)/2}).$$  Then there
must be a point $x\in M$ with $|{\mathcal R}_x|>0$.  \end{theorem}

Since ${\mathcal R}_x\subset {\mathcal L}_x$ this result is stronger than the first part
of Theorem \ref{QMB}.

The proof of Theorem \ref{quasimodetheorem} is based on Theorem
\ref{RECURRENCE} and the following lemma.

\begin{lemma}\label{lemma2}  Fix $B>0$ and
suppose that  for $\lambda=\lambda_j\to \infty$ we have
\begin{equation}\label{omega}
\|\psi_\lambda\|_\infty \ge B\lambda^{(n-1)/2}.
\end{equation}  Then if $0<\delta<1$ there
exists $\varepsilon>0$ so that if $\lambda=\lambda_j$ and
\begin{equation}\label{small}
\|(\Delta+\lambda^2)\psi_\lambda\|_2+\|S_{2\lambda}^\perp
\psi_\lambda\|_\infty \le \varepsilon \lambda,
\end{equation}
then if
$\chi_{[\lambda-\delta,\lambda+\delta]}f=\sum_{\lambda_j\in
[\lambda-\delta,\lambda+\delta]}e_j(f)$,
\begin{equation}\label{*}
\|\chi_{[\lambda-\delta,\lambda+\delta]} \psi_\lambda\|_\infty \ge
\frac{B}2 \lambda^{(n-1)/2},
\end{equation}
for all sufficiently large $\lambda=\lambda_j$.
\end{lemma}

Before proving Lemma \ref{lemma2}, let us see why it and Theorem
\ref{RECURRENCE} implies Theorem \ref{quasimodetheorem}. To do this, let us suppose
that we have a sequence of admissible quasimodes satisfying
$\|\psi_\lambda\|_\infty=\Omega(\lambda^{(n-1)/2})$  If we apply
Lemma \ref{lemma2} we conclude that there is a positive constant
$c>0$ so   that for any $0<\delta<1$ we have
$$\|\chi_{[\lambda-\delta,\lambda+\delta]}\psi_\lambda\|_\infty \ge
c\lambda^{(n-1)/2},$$ for some sequence $\lambda=\lambda_j$, if
$\lambda$ is large enough (depending on $\delta$).

Let $\rho>0$.  If there were no recurrent points, we could apply
Theorem \ref{RECURRENCE} to conclude that there is a
$\delta=\delta(\rho)$ so that for large enough $\lambda$
(depending on $\rho$)
$$\|\chi_{[\lambda-\delta,\lambda+\delta]} \psi_\lambda\|_\infty\le
C\rho\lambda^{(n-1)/2},$$ which leads to a contradiction if $\rho$
is chosen small enough so that $C\rho<c$. Thus, we conclude that
there must be a recurrent point under the hypotheses of  Theorem
\ref{quasimodetheorem}.  \qed

\noindent{\bf Proof of Lemma \ref{lemma2}:}  To simplify the
notation, let us set $\chi^\delta_\lambda =
\chi_{[\lambda-\delta,\lambda+\delta]}$.  We need to see under the
hypotheses of Lemma \ref{lemma2} we have
\begin{equation}\label{!}
\|(I-\chi_\lambda^\delta)\psi_\lambda\|_\infty \le \frac{B}2
\lambda^{(n-1)/2},
\end{equation}
if $\lambda=\lambda_j$ is large enough.

This would follow from a couple of estimates.  The first one says
that there is a constant $A$ which is independent of $0<\delta<1$
and $\lambda>1$ so that
\begin{equation}\label{!!}
\|\chi_\lambda^1(I-\chi_\lambda^\delta)f\|_\infty \le
A\lambda^{(n-1)/2}(\lambda\delta)^{-1}\|(\Delta+\lambda^2)f\|_2,
\end{equation}
while the second one says that
\begin{equation}\label{!!!}
\|(I-\chi_\lambda^1)S_{2\lambda}f\|_\infty\le
C\lambda^{(n-1)/2}\lambda^{-1}\|(\Delta+\lambda^2)f\|_2.
\end{equation}

To see how these imply \eqref{*}, we take $f=\psi_\lambda$.  Then
since $\delta<1$ we have
\begin{equation}\label{!!!!}(I-\chi^\delta_\lambda)\psi_\lambda =
\chi^1_\lambda(I-\chi^\delta_\lambda)\psi_\lambda+(I-\chi^1_\lambda)S_{2\lambda}\psi_\lambda
+S^\perp_{2\lambda}\psi_\lambda.
\end{equation}

If $n\ge4$ we estimate the last piece by the second part of our
admissible quasimode hypothesis
$\|S^\perp_{2\lambda}\psi_\lambda\|_\infty=o(\lambda)=o(\lambda^{(n-1)/2})$.
If $n\le 3$ we use Sobolev to get that for a given $0<\sigma<1/2$
\begin{align*}
\|S^\perp_{2\lambda}\psi_\lambda\|_\infty &\le
C\|(\sqrt{-\Delta})^{n/2+\sigma}S^\perp_{2\lambda}\psi_\lambda\|_2
\\
&\le
C\|(\sqrt{-\Delta})^{\tfrac{n}2-2+\sigma}S^\perp_{2\lambda}(\Delta+\lambda^2)\psi_\lambda\|_2
\\
&\le
C\lambda^{\tfrac{n}2-2+\sigma}\|(\Delta+\lambda^2)\psi_\lambda\|_2
\\
 &\le
C\lambda^{\tfrac{n}2-2+\sigma}\lambda=o(\lambda^{(n-1)/2}),
\end{align*}
as desired since $\sigma<1/2$.

Using \eqref{small}, \eqref{!!} and \eqref{!!!} we can estimate
the remaining pieces in \eqref{!!!!}
\begin{align*}
\|\chi^1_\lambda(I-\chi^\delta_\lambda)\psi_\lambda\|_\infty+\|(I-\chi^1_\lambda)S_{2\lambda}\psi_\lambda\|_\infty
&\le (A+C)\lambda^{(n-1)/2}
(\lambda\delta)^{-1} \varepsilon (\lambda/\log\lambda)
\\
&\le 2(A+C)(\varepsilon/\delta)\lambda^{(n-1)/2},
\end{align*}
if $\lambda$ is large.  Since this estimate and our earlier bounds
for $S_{2\lambda}\psi_\lambda$ yield \eqref{!}, we are left with
proving \eqref{!!} and \eqref{!!!}.

The estimate \eqref{!!} is easy.  Using the fact that
$\|\chi_\lambda^1\|_{L^2\to L^\infty}\le A\lambda^{(n-1)/2}$, we
get
\begin{align*}
\|\chi^1_\lambda(I-\chi^\delta_\lambda)f\|_\infty &\le
A\lambda^{(n-1)/2}\|(I-\chi^\delta_\lambda)f\|_2
\\
&\le
A\lambda^{(n-1)/2}\|(\Delta+\lambda^2)^{-1}(I-\chi_\lambda^\delta)(\Delta+\lambda^2)f\|_2
\\
&\le
A\lambda^{(n-1)/2}(\lambda\delta)^{-1}\|(\Delta+\lambda^2)f\|_2.
\end{align*}

To prove \eqref{!!!}, let $\Pi_{[j,j+1)}$ denote the projection
onto the $[j,j+1)$ part of the spectrum of $\sqrt{-\Delta}$.  Then
we can write
$$(I-\chi^1_\lambda)S_{2\lambda}f=
\sum_{k=1}^\lambda\Bigl(\,
\Pi_{[\lambda+k,\lambda+k+1)}S_{2\lambda}f+\Pi_{[\lambda-k-1,\lambda-k)}S_{2\lambda}f\,
\Bigr)
,
$$
Thus,
\begin{align*}
\|(I-\chi^1_\lambda)S_{2\lambda}f\|_\infty&\le
\sum_{k=1}^\lambda\Bigl(\|\Pi_{[\lambda+k,\lambda+k+1)}S_{2\lambda}f\|_\infty
+\|\Pi_{[\lambda-k-1,\lambda-k)}S_{2\lambda}f\|_\infty\Bigr)
\\
&=I+II.
\end{align*}
We shall only estimate $I$ since the same argument will yield the
same bounds for $II$.

To estimate $I$, we first note that for $1\le k\le \lambda$
\begin{align*}
\|\Pi_{[\lambda+k,\lambda+k+1)} g\|_\infty &\le
C\lambda^{(n-1)/2}\|\Pi_{[\lambda+k,\lambda+k+1)} g\|_2
\\
&=C\lambda^{(n-1)/2}\|(\Delta+\lambda^2)^{-1}\Pi_{[\lambda+k,\lambda+k+1)}
(\Delta+\lambda^2)g\|_2
\\
&\le C\lambda^{(n-1)/2}(\lambda
k)^{-1}\|\Pi_{[\lambda+k,\lambda+k+1)} (\Delta+\lambda^2)g\|_2
\end{align*}
Therefore, by applying the Schwarz inequality, we get
\begin{align*}
I&=\sum_{k=1}^\lambda k^{-1} \bigl(\,
k\|\Pi_{[\lambda+k,\lambda+k+1)} S_{2\lambda}f\|_\infty\, \bigr)
\\
&\le C\lambda^{(n-1)/2}\lambda^{-1}\Bigl(\sum_{k=1}^\lambda
\|\Pi_{[\lambda+k,\lambda+k+1)}
(\Delta+\lambda^2)f\|_2^2\Bigr)^{1/2}
\\
&\le C\lambda^{(n-1)/2}\lambda^{-1}\|(\Delta+\lambda^2)f\|_2,
\end{align*}
as desired.  Since, as we noted, the same argument works for $II$,
we have completed the proof of Lemma \ref{lemma2}. \qed

Let us conclude this section by pointing out that the conclusion
of the lemma is not valid for dimensions $n\ge 4$ if one just
assumes $\|(\Delta+\lambda_k^2)\psi_{\lambda_k}\|_2=o(\lambda_k)$
or even
\begin{equation}\label{I}\|(\Delta+\lambda_k^2)\psi_{\lambda_k}\|_2=O(1)
\end{equation}
for the quasimode definition.

Let us first handle the case where $n\ge5$ since that is slightly
simpler than the $n=4$ one.  To handle this case, we fix a
nonnegative function $\eta\in C^\infty_0({\mathbb{R}})$ satisfying
$\eta(10)=1$ and $\eta(s)=0$, $s\notin [5,20]$.  We then set
$$K=\lambda_k^s,
$$
where $s>1$ is large and will be chosen later.  Put
$$\psi_{\lambda_k}(x)=K^\varepsilon
K^{-n}\eta(\sqrt{\Delta}/K)(x_0,x)=K^\varepsilon
K^{-n}\sum_{\lambda_j}\eta(\lambda_j/K)e_{\lambda_j}(x_0)e_{\lambda_j}(x),$$
where $x_0\in M$ is fixed and $\varepsilon>0$ is small.

We notice that the conclusion of the lemma is false for these
functions since
$\chi_{[\lambda_k-\delta,\lambda_k+\delta]}\psi_{\lambda_k}\equiv
0$ if $s>1$ is fixed and if $\lambda_k$ is large, due to the fact
that the spectrum of $\psi_{\lambda_k}$ is in $[5\lambda_k^s,
20\lambda_k^s]$ and $\lambda_k$ does not lie in this interval for
large $k$. Also, it is not hard to verify that
$$\psi_{\lambda_k}(x_0)\approx K^\varepsilon =\lambda_k^{s\varepsilon},$$
and so by choosing $s=\frac{n-1}{2\varepsilon}$, we have one of
the assumptions of the lemma that $\|\psi_{\lambda_k}\|_\infty =
\Omega(\lambda_k^{(n-1)/2})$.  We also have \eqref{I} if $n\ge5$.
For then
$$\|(\Delta+\lambda_k^2)\psi_{\lambda_k}\|_2\approx \|(\Delta+1)\psi_{\lambda_k}\|_2\approx
\lambda_k^2\|\psi_{\lambda_k}\|_2 \approx K^2K^\varepsilon
K^{-n}K^{n/2}=o(1),$$ if, as we may, we choose $\varepsilon<1/2$.

Minor modifications of this argument show that things break down
for $n=4$ as well if one just assumes \eqref{I}.   Here one would
take $j_0=2^{\lambda_k^{n-1}}$ so that $\log j_0=\lambda_k^{n-1}$,
where $\log$ is the base-2 log.  Then, with the above notation,
one sets
$$\psi_{\lambda_k}(x)=(\log j_0)^{-1/2}\sum_{j\in[\log j_0, 2\log
j_0]}2^{-jn}\eta(\sqrt{\Delta}/2^j)(x_0,x).$$ Then, one can see
that
$$\psi_{\lambda_k}(x_0)\approx (\log j_0)^{1/2}=\lambda_k^{(n-1)/2},$$
$\chi_{[\lambda_k-\delta,\lambda_k+\delta]}\psi_{\lambda_k}\equiv
0$ if $\lambda_k$ is large. Finally, if $n=4$, \eqref{I} is valid
since
\begin{multline*}
\|(\Delta+\lambda_k^2)\psi_{\lambda_k}\|_2^2\approx
\|(\Delta+1)\psi_{\lambda_k}\|_2^2
\\
\approx (\log j_0)^{-1}\sum_{j\in [\log j_0, 2\log
j_0]}2^{4j}2^{-8j}\|\eta(\sqrt{\Delta}/2^j)(x_0,\cdot)\|_2^2
\approx (\log j_0)^{-1}\sum_{j\in [\log j_0, 2\log j_0]} 1 \approx
1.
\end{multline*}

These constructions will also show that when $n\ge4$ one cannot
use
$$\|(\Delta+\lambda_k^2)\psi_{\lambda_k}\|_2=O(\lambda_k^{-s})$$
for any large $s$ as the condition for quasimodes
$\{\psi_{\lambda_k}\}$ and have the conclusions of the lemma be
valid.

\subsection{Quasi-modes associated to blow-down points: Proof of
Theorem \ref{QMB} (2):} We now prove the  converse result in
Theorem \ref{QMB} under the assumption that $G^T_z = Id$ . The
method is to construct  quasi-modes associated to the ``blow-down''
 Lagrangian
$\Lambda_z$ in Definition \ref{BDDEF} (see below). The analysis generalizes
the one in \cite{Z} in the Zoll case. The key point is the
existence of an invariant 1/2-density on $\Lambda_z$ for the
geodesic flow.  In this case, the invariant 1/2-density is $|d\mu
\, dt|^{1/2}$ where, $d\mu$ is Liouville measure on
 $S_z^*M$.

\subsubsection{The Blow-down Lagrangian.}
Since $z$ is a blow-down point, the geodesic flow induces a smooth
first return  map (\ref{RM}).  Let ${\mathcal C}_T$ denote the
mapping cylinder of $G^T_z$, namely
\begin{equation} {\mathcal C}_T = S_z^*M \times [0, T] /
\cong,\;\;\;\mbox{where}\;\; (\xi, T) \cong (G^T_z(\xi), T).
\end{equation}
The ${\mathcal C}_T$ is a smooth manifold. It naturally fibers
over $S^1$ by the map
$$\pi : {\mathcal C}_T \to S^1,\;\; \pi(\xi, t) = t \; \mod\; 2
\pi \Z. $$

\begin{prop} \label{BDDEF} Let $(M, g)$ be an $n$-dimensional Riemannian manifold,
 and assume that it possesses a blow down point
$z$. Let  $\iota_{z}: {\mathcal C}_T \to T^*M$ be the map
$$\iota_z(\xi, t) = G^t (z, \xi). $$
Then $\iota_z$ is a Lagrange embedding whose image is a
geodesic-flow invariant Lagrangian manifold, $\Lambda_{z}$, diffeomorphic to $S^1
\times S^{n-1}  \simeq {\mathcal C}_T$.
\end{prop}

\begin{proof}

We let $\omega$ denote the canonical symplectic form on $T^*M$.
Then,  under the map \begin{equation} \label{IOTA} \iota: S^1
\times S_x^*M \to T^*M, \;\; \iota (t, x, \xi) \to G^t(x, \xi),
\end{equation} we have
$$\iota^* \omega = \omega - dH \wedge dt, \;\; H(x, \xi) = |\xi|_g. $$
The map $\iota_z$ is the restriction of $\iota$ to $\R \times
S^*_z M$. Since $dH = 0$ on $S^*M$ and $\omega = 0$ on $S_x^*M,$
the right side equals zero.

Thus, $\iota_x$ is a Lagrange immersion. To see that it is an
embedding, it suffices to prove that it is injective, but this is
clear from the fact that $G^t$ has no fixed points.

\end{proof}

Let $\alpha_{\Lambda}$ denote the action form $\alpha = \xi \cdot
dx$ restricted to $\Lambda$. Also, let $m_{\Lambda}$ denote the
Maslov class of $\Lambda.$ A Lagrangian $\Lambda$ satisfies the
Bohr-Sommerfeld quantization condition \cite{D} if
\begin{equation} \label{BS} \frac{r_k}{2 \pi} \left[
\alpha_{\Lambda} \right] \equiv \frac{m_{\Lambda}}{4} \;\;
\mbox{mod}\;\; H^1(\Lambda, \Z),\end{equation} where
$$r_k = \frac{2 \pi}{T} (k + \frac{\beta}{4}), $$
with $\beta$ equal to the common Morse index of the geodesics
$G^t(z, \xi), \xi \in S^*_z M.$

\begin{prop} \label{Bohr} $\Lambda_z$ satisfies the Bohr-Sommerfeld
quantization condition. \end{prop}

\begin{proof} We need to identify the action form and Maslov
class.

\begin{lem}  We have:

\begin{enumerate}

\item $\iota_z^* \alpha_{\Lambda} = dt$.

\item $\iota_z^* m_{\Lambda_z} = \frac{\beta}{T} [dt]. $
\end{enumerate}

 \end{lem}

\begin{proof}

\noindent{\bf (1)} Let $\xi_H$ denote the Hamiltonian vector field
of $H$. Since $(G^t)^* \alpha = \alpha$ for all $t$, we may
restrict to $t = T$ and to $S^*_z M$ to obtain $(G^T_z)^*
\alpha|_{S^*_z M} = \alpha|_{S^*_z M}. $ But clearly, $\xi \cdot
dx |_{S^*_z M} = 0$.
\medskip

\noindent{\bf (2)} We recall that $m_{\Lambda_z} \in
H^1(\Lambda_z, \Z)$ gives the oriented intersection class with the
singular cycle $\Sigma \subset \Lambda_z$ of the projection $\pi:
\Lambda_z \to M. $ Given a closed  curve $\alpha$ on $\Lambda_z$,
we deform it to intersect $\Sigma$ transversally and then
$\int_{\alpha} m_{\Lambda_z}$ is the oriented intersection number
of the curve with $\Sigma.$ Our  claim is that $\int_{\alpha}
m_{\Lambda_z} = \beta$ where $\beta$ is the common Morse index of
the (not necessarily smoothly) closed geodesic loops
$\gamma_{\xi}(t) = G^t(z, \xi), \xi \in S^+_z M.$

The inverse image of the singular cycle of $\Lambda_z$ under
$\iota_z$ consists of the following components:
$$\iota_z^{-1} \Sigma = S^*_z M \cup Conj(z), $$
where $$Conj(z) = \{(t, \xi): 0 < t < T,\; \xi \in S^*_z,\;\;
|\det d_z \exp t \xi | = 0\}$$ is the tangential conjugate locus
of $z$. All of $ S^*_z M$ consists of self-conjugate vectors at
the time $T$.

If $\dim M \geq 3$, then $H^1({\mathcal C}_T, \Z) = \Z$ is
generated by the homology class of a  closed geodesic loop at $z$
and in this case $\int_{\alpha} m_{\Lambda_z} = \beta$ by
definition of the Morse index. If $\dim M = 2$, then
$H^1({\mathcal C}_T, \Z)$ has two generators, that of a closed
geodesic loop and that of $S^*_z M.$ The value of $m_{\Lambda_z}$
on the former is the same as for $\dim M \geq 3,$ so it suffices
to determine $\int_{S^*_z M} m_{\Lambda_z}. $  To calculate the
intersection number, we deform $S^*_z M$ so that it intersects
$\iota_z^{-1} \Sigma$ transversally. We can use $G^{\epsilon}
S^*_z M$ as the small deformation, and observe that it has empty
intersection with $\iota_z^{-1} \Sigma$ for small $\epsilon$ since
the set of conjugate times and return times have  non-zero lower
bounds.

\end{proof}

The Lemma immediately implies  (\ref{BS}), completing the proof.

\end{proof}

\subsubsection{Construction of quasi-modes.}

We now `quantize' $\Lambda_z$ as a space of oscillatory integrals.

\begin{lem} \label{lagsymbol} There exists $\Phi_k \in {\mathcal O}^{\frac{n-1}{2}} (M, \Lambda_z,
\{r_k\})$ with $\iota^* \sigma(\Phi_k) = e^{-
i r_k t} |dt|^{1/2} \otimes |d \mu|^{1/2}, $ where $d\mu$ is
Liouville measure on $S^*_zM.$ \end{lem}

We will refer to $\Phi_k$ as quasi-modes associated to the  blow
down point  $z \in M$.
\medskip

\noindent{\bf Examples}
\begin{enumerate}

\item In the case of $S^n$ and $z$ the north pole,  $\Phi_k(z)$ is
the zonal spherical harmonic of degree $k$. Equivalently, it
equals, up to $L^2$-normalization, the orthogonal projection
kernel $\Pi_k(\cdot, z)$ onto $k$th order with second variable
fixed at $z$. In this case, it is an eigenfunction.

\item On a general Zoll manifold, with $z$ any point, the
projection kernel onto the $k$th eigenvalue cluster is a
quasi-mode of this type, see \cite{Z}. In general, it is a zeroth
order quasi-mode, reflecting the width $k^{-1}$ of the $k$th
cluster, and not an eigenfunction.

\item On a surface of revolution diffeomorphic to $S^2$, the zonal
eigenfunctions are oscillatory integrals of this type.

\end{enumerate}

 \subsection{   Proof of Theorem \ref{ergsup}.}
By the results of \cite{SZ},  $(M, g)$ possesses a point $m$  such
that all geodesics issuing from the point $m$ return to $m$ at
some time $\ell$ (which with no loss of generality may be taken to
be $2\pi$).  By Proposition \ref{BDDEF}, the map $\iota$ of
(\ref{IOTA}) is a Lagrange immersion with  image $\Lambda_m$.

If $\dim M = 2,$  the image $\iota([0, 2 \pi] \times S^*_m M)$ is
a Lagrangian torus,  the mapping torus of the first return map
$G^{2 \pi} |_{S^*_m M } : S^*_m M \to S^*_m M. $ Obviously, $G^t
(\Lambda) = \Lambda $ for all $t$, so $\Lambda$ is an invariant
torus for the geodesic flow. Moreover, $M$ is diffeomorphic to
$S^2$ or to $\R P^2$.  Since $S^*M = \R P^3$ when  $M = S^2$ (or
in the case $ \R P^2$ is a  quotient by a $\Z_2$ action), we have
$H^2(S^*M) = \{0\} $.  Hence, $\Lambda =
\partial \Omega$  where $\Omega \subset S^*M$ is a singular
$3$-chain. Since $\dim S^*M = 3$, $\Omega$ has a non-empty
interior, so $\Lambda$ is the boundary of an open set. But  $G^t
\Omega \subset \Omega$. Hence, there exists an open invariant set,
and $G^t$ cannot be ergodic. \qed

In higher dimensions, we do not see how ergodicity rules out
existence of invariant Lagrangian $S^1 \times S^{m-1}$ or blow
down points. Hyperbolicity of the geodesic flow is inconsistent
with existence of such Lagrangian submanifolds. But, as mentioned
in the introduction, there are better estimates in the case of
$(M, g)$ with Anosov (hyperbolic) geodesic flows. These  never
have conjugate points, and  the generic sup norm estimate can be
improved to $||\phi_j||_{L^{\infty}} =
O(\lambda_j^{\frac{n-1}{2}}/\log \lambda_j) $ (\cite{Be}).  But
 of course such flows do not exist for metrics on  $S^2$, and  the  previous result
provides new information for analytic metrics with  ergodic
geodesic flow on $S^2$.

\subsubsection{Pointwise asymptotics of the quasimode $\Phi_k$. }

 In the following we let $\hbar \in \{r_k\}^{-1};k=1,2,...$ and let $B_j \subset M; j =1,...,N$ be small geodesically convex balls with $ \pi ( \Lambda_z) \subset \cup_{j=1}^{N} B_j.$  Let $\chi_{j} \in C^{\infty}_{0}(B_j)$ be a partition of unity subordinate to this covering and $\chi_{R}(s) \in C^{\infty}_{0}({\mathbb R})$ be a cutoff equal to $1$ when  $ |s| < R$ with $R > 1$ and zero when $|s| > 2R.$ One then constructs the quasimode $\Phi_{k}(x)$ as a sum $\sum_{j=1}^{N} \chi_j \Phi_{k}^{(j)}$ where the $\Phi_{k}^{(j)} \in C^{\infty}(B_j)$ are local oscillatory integrals of the form

\,\,$\Phi_k^{(j)}(x) = (2\pi \hbar)^{\frac{1-n}{2}} \int_{{\Bbb R}^{n}} e^{i
\phi^{(j)}(x,\theta)/\hbar} \, a^{(j)}(x,\theta;\hbar)   \, \chi_{R}(|\theta|) \,
d\theta.$

Without loss of generality, we assume that $z \in B_{1}$ and  let
 $x= (x_{1},...,x_{n}) \in B_1$  be geodesic  normal coordinates with $x(z)=0 \in {\Bbb R}^{n}$.
Consider first
  $$\Phi_{k}^{(1)}(x) = (2\pi \hbar)^{\frac{1-n}{2}}\int_{{\Bbb R}^{n}} e^{i \phi^{(1)} (x,\theta)/\hbar} \, a^{(1)} (x,\theta;\hbar) \, \chi_{R}(|\theta|) \, d\theta.$$

The $L^{2}$-normalized quasimode $\Phi_{k}$  is constructed to solve the  equation $ \|-\Delta_{g} \Phi_{k} - r_k^2 \Phi_{k} \|_{L^2} = {\mathcal O}(1) $ and for this, one needs to globally solve the eikonal equation and the first transport equation.

For the eikonal equation, we  choose the phase $\phi^{(1)} = \phi^{(1)}(x,\theta)$ positive homogeneous of
degree zero in the $\theta_{j}$-variables. Since $S_{z}^{*}M
\subset \Lambda_{z} \cap  \pi^{-1}(B_{1})$ is non-characteristic for the
geodesic flow, it follows that there exists a locally unique
solution $\phi^{(1)}(x,\theta)$ to the initial value problem
\begin{equation} \label{eikonal1}
 |\nabla_{x} \phi^{(1)}(x,\theta)|^{2}_{g} = 1
\end{equation}
\begin{equation} \label{eikonal2}
\phi^{(1)}(0,\theta) = 0,
\end{equation}
with
\begin{equation} \label{par}
 \Lambda_{z} \cap \pi^{-1}( B_{1}) =
\{(x,\partial_{x} \phi^{(1)}(x,\theta)) \in B_{1} \times {\mathbb R}^{n}; \,
\partial_{\theta}\phi^{(1)}(x,\theta) = 0 \}. \end{equation}
Consider the function
$$ \phi^{(1)}(x,\theta) =  \langle x, \frac{\theta}{|\theta|} \rangle, \,\,\, \theta \neq 0.$$
By the Gauss lemma,
\begin{equation} \label{gauss}
 \sum_{j=1}^{n} g_{ij}(x) x_{j} = \sum_{j=1}^{n} g_{ij}(0) x_j = x_i,
\end{equation}
 and so, $\langle x, \theta \rangle_{g} = = \sum_i x_{i} \theta_{i}$.  Consequently, for $\theta \neq 0$ we have that
\begin{equation} \label{eik1}
\phi^{(1)}(x,\theta) = \langle x, \frac{\theta}{|\theta|} \rangle = \langle x,\frac{\theta}{|\theta|_g} \rangle_g.
\end{equation}
Then, from (\ref{gauss}) it follows that $\phi^{(1)}(0,\theta) = 0$ and
\begin{equation} \label{eik2}
|\nabla_{x} \phi^{(1)}(x,\theta)|_{g} ^{2} = \sum_{i,j=1}^{n} g^{ij}(x)
\partial_{i}\phi \partial_j \phi^{(1)} =
\frac{|\theta|_g^{2}}{|\theta|_g^{2}} =1.
\end{equation}
Thus, $\phi^{(1)}(x,\theta) = \langle x, \frac{\theta}{|\theta|} \rangle$ satisfies the initial value problem in (\ref{eikonal1}) and (\ref{eikonal2}). Moreover, a direct computation shows that
$$\{(x,\partial_{x} \phi^{(1)}(x,\theta)) \in B_{1} \times {\mathbb R}^{n}; \,
\partial_{\theta}\phi^{(1)}(x,\theta) = 0 \} = \{ (t \omega, \omega) \in {\mathbb R}^{n} \times {\mathbb S}^{n-1}; |t| < \epsilon_0 \}.$$ Here, $\epsilon_0$ is the geodesic radius of the ball $B_1.$  The latter set is just $\Lambda_z \cap \pi^{-1}(B_1)$ written in normal coordinates.



The  transport equation for  $a_{0}^{(1)}(x,\theta)$ is
\begin{equation} \label{transport1}
 g^{ij} \partial_{x_i} \phi \cdot \partial_{x_j} a_{0}^{(1)} = g^{ij} \partial_{x_{i}} \partial_{x_{j}}\phi \cdot a_{0}^{(1)} = g^{ij} \partial_{x_{i}} \partial_{x_{j}} ( \langle x, \theta \rangle)  \cdot a_{0}^{(1)} =0,
\end{equation}

where, we impose the initial condition $a_{0}^{(1)}(0,\theta) =1.$
It follows that
\begin{equation} \label{amplitude}
a_{0}^{(1)}(x,\theta) = 1.
\end{equation}

\subsubsection{$L^2$-normalization}
Consider first the local quasimode $\Phi_{k}^{(1)}$ and  choose \newline  $\delta \in (1- \frac{1}{n},1)$. Clearly,
\begin{equation} \label{junk}
 \int_{|x| \leq \hbar^{\delta}} |\Phi_{k}^{(1)}(x)|^{2} \, dx  = {\mathcal O}( \hbar^{ 1 - (1-\delta)n }). \end{equation}

In the annulus $ A_{\delta}(\hbar):= \{ x \in B_1; \hbar^{\delta} < |x| < \epsilon_0 \},$ we introduce polar coordinates and  write
$$\Phi_{k}^{(1)}(x) = (2\pi \hbar)^{\frac{1-n}{2}} \int_{{\mathbb R}^{n}} e^{i \frac{|x|}{\hbar} \langle \frac{x}{|x|}, \frac{\theta}{|\theta|} \rangle }  \chi_{R}(|\theta|) \, d\theta $$
\begin{equation} \label{stationaryphase}
= (2\pi \hbar)^{\frac{1-n}{2}} \, \int_{0}^{\infty}  \left( \int_{{\mathbb S}^{n-1}} e^{ i \frac{|x|}{\hbar}  \langle  \frac{x}{|x|}, \omega \rangle } \,  d\omega \right) \, \chi_{R}(r) \, r^{n-1} \, dr. \end{equation}
Since $\frac{|x|}{\hbar} \rightarrow \infty$ as $\hbar \rightarrow 0^+$, one makes a stationary phase expansion in the inner $\omega$-integral in (\ref{stationaryphase}). The result is that  for $x \in A_{\delta}(\hbar),$
\begin{equation} \label{stationaryphase2}
\Phi_{k}^{(1)}(x) =  |x|^{ \frac{1-n}{2}}  (  \, c_+ \, e^{i \frac{|x|}{\hbar}}   \, + \, c_{-} \, e^{- i \frac{|x|}{\hbar}} \,  + {\mathcal O}( |x|^{-1} \hbar) \, ) \end{equation}
Here, $c_{\pm} \in {\mathbb C}$ with $|c_{\pm}| \neq 0.$ It follows from (\ref{stationaryphase2}) that
\begin{equation} \label{lead}
\int_{ A_{\delta}(\hbar)} |\Phi_{k}^{(1)}(x)|^{2} \, dx = ( |c_{+}|^{2} + |c_{-}|^{2}) \epsilon_{0} +  {\mathcal O}(\hbar^{\delta}). \end{equation}
From (\ref{lead}) and (\ref{junk}) it follows that  there is a constant $C(\epsilon_{0}) >0$ such that for $\hbar$ sufficiently small,
\begin{equation} \label{l2norm}
\int_{B_1} |\Phi_{k}^{(1)}(x)|^{2} dx = C(\epsilon_{0}) + {\mathcal O}(\hbar^{\delta'}),  \,\,\,\, \delta' = \min \, ( 1- (1-\delta)n, \delta ).\end{equation}
The computation for the other quasimodes is the same and so, there exist constants $C_{j} >0, j=2,...,N$ such that   also $\| \Phi_{k}^{(j)} \|_{L^2} \sim C_{j}$ for all $j \neq 1.$
Since $\Lambda_{z}$ satisfies the Bohr-Sommerfeld quantization conditions  in Proposition \ref{Bohr}, the local quasimodes satisfy $\Phi_{k}^{(m)}(x) = \Phi_{k}^{(m')}(x)$ for all $x \in B_{m} \cap B_{m'}$ and so they patch together to form a global  quasimode $\Phi_{k}.$ After possibly multiplying $\Phi_{k}$ by  a postiive constant, it follows that  $\| \Phi_{k} \|_{L^2} \sim 1$ with  $\| (-\Delta_g - r_k^2) \Phi_k \|_{L^{2}} = {\mathcal O}(1)$ as $k \rightarrow \infty.$

\subsubsection{Symbol computations} In normal coordinates, the map $\iota_z$ is  given by the formula
$$ \iota_z (t,\omega) = (t \omega, \omega); \,\, t \in {\mathbb R}/[0,T].$$
Let $\iota_{\phi}: C_{\phi} \rightarrow \Lambda_{z}$ be the standard immersion $ (x,\theta) \mapsto (x,\partial_{x} \phi), \,\, (x,\theta) \in C_{\phi}$ where $C_{\phi}:= \{ (x,\theta) \in M \times {\mathbb R}^{n}, \, \partial_{\theta} \phi(x,\theta) = 0 \}.$ Then,
$$ \iota_z^{*} (\iota_{\phi}^{-1})^* \phi (t,\omega) = \langle t \omega, \omega \rangle = t, \,\,\, \omega \in {\mathbb S}^{n-1}.$$
This is the phase function of the principal symbol $ \iota_z^* \sigma(\Phi_{k})$ in Proposition \ref{lagsymbol}. For the amplitude of the symbol  $\iota_z^* \sigma(\Phi_k),$ one looks for a half-density solution $\tilde{a} \in C^{\infty}({\mathcal C}_{T}; |\Omega|^{ \frac{1}{2} })$ of  the equation
$$ \frac{d}{ds} G_{s}^* \tilde{a}(t,\omega)|_{s=0} = 0,$$
and in view of  (\ref{amplitude}),  the required solution is given by
\begin{equation} \label{symbolamplitude}
 \tilde{a}(t,\omega) =  (2\pi \hbar)^{\frac{1-n}{2}} | dt d\mu_{\omega} |^{\frac{1}{2}}. \end{equation}

Consequently, $\iota_{z}^* \sigma(\Phi_{k})(t,\omega) = (2\pi r_{k})^{\frac{n-1}{2}}  e^{i t r_{k}}  |d\mu_{\omega} dt |^{\frac{1}{2}}$ as in Lemma \ref{lagsymbol} and  moreover, we have proved



\begin{prop} \label{zonal}
Let $\Phi_k \in {\mathcal O}^{\frac{n-1}{2}} (M,\Lambda_{z}, \{r_k \})$
be the $L^{2}$-normalized quasimode constructed above.  Then,  $\iota_{z}^{*}  \sigma(\Phi_{k})(t,\omega) = (2\pi r_k)^{\frac{n-1}{2}} e^{-it r_{k}} \, |dt|^{\frac{1}{2}} \otimes |d\mu_{\omega}|^{\frac{1}{2}}$ and
$$ |\Phi_k (z)| =  (2\pi r_{k})^{ \frac{n-1}{2} } \int_{{\Bbb R}^{n}} a^{(1)}_{0}(0,\theta ;\hbar)  \,  \chi_{R}(|\theta|) \,d\theta  \sim_{k \rightarrow \infty}  C_{R} r_{k}^{ \frac{n-1}{2} }.$$
Here, $C_{R} >0$ is a constant depending only on $R.$
\end{prop}

\noindent This completes the proof of Theorem \ref{QMB} (ii). \qed

\end{document}